\documentclass[11pt,reqno]{amsart}
\usepackage[utf8]{inputenc}
\usepackage[margin=1in]{geometry}
\usepackage{setspace}
\usepackage{parskip}
\usepackage{amsmath}
\usepackage{amssymb}
\usepackage{amsthm}
\usepackage{mathtools}
\usepackage[mathscr]{euscript}
\usepackage{bbm}
\usepackage{braket}
\usepackage{arydshln}
\usepackage{shortcutpackagev11}
\usepackage{graphicx}
\graphicspath{{./images/}}
\usepackage{tikz}
\usetikzlibrary{cd}
\usepackage{quiver}
\usepackage{pgfplots}
\pgfplotsset{width=10cm, compat=1.9}
\usepackage{url}
\usepackage{verbatim}
\usepackage{booktabs}
\usepackage{longtable}
\usepackage{array}
\usepackage[colorlinks=true, linkcolor=black, citecolor=black, urlcolor=cyan]{hyperref}
\usepackage[
  backend=biber, style=alphabetic,
  sorting=nyt, maxnames=99
]{biblatex}
\usepackage{thm-restate}

\makeatletter
\renewcommand{\section}{\@startsection{section}{1}%
  \z@{.7\linespacing\@plus\linespacing}{.5\linespacing}%
  {\normalfont\bfseries\centering}}
\makeatother

\newcommand{\abs}[1]{\left|#1\right|}

\newtheorem{theorem}{Theorem}[section]

\newtheorem{lemma}[theorem]{Lemma}
\newtheorem{corollary}[theorem]{Corollary}
\theoremstyle{plain}
\newtheorem{definition}[theorem]{Definition}

\theoremstyle{plain}

\newtheorem{remark}[theorem]{Remark}

\numberwithin{equation}{section}
\title{Bounding The Number of Zeros Near the Central Point in Families of Cuspidal Newforms}

\author{Lucas Chen}
\address{Department of Mathematics, University of Chicago}
\email{\href{mailto:}{lucasch@uchicago.edu}}

\author{Christopher Housholder}
\address{Department of Mathematics, Missouri State University}
\email{\href{mailto:}{christopherlancehousholder@gmail.com}}

\author{Joshua Khan}
\address{Department of Mathematics, University of Rochester}
\email{\href{mailto:}{joshuakhan1307@gmail.com}}

\author{Steven J. Miller}
\address{Department of Mathematics, Williams College}
\email{\href{mailto:}{sjm1@williams.edu}}

\author{Devayani Pradhan}
\address{Department of Mathematics, University of Michigan}
\email{\href{mailto:}{pradhand@umich.edu}}

\begin{document}
\begin{spacing}{1.05}

\begin{abstract}
We study low-lying zeros in families of even holomorphic cuspidal newforms of fixed weight and prime level, with particular emphasis on the number of forms having a zero in a prescribed normalized window about the central point and on the distribution of the number of such zeros among the forms. We quantify the number of forms having at least one zero in the window and study the distribution of the number of zeros in that window among the forms. Assuming the Generalized Riemann Hypothesis, we obtain new lower bounds for the number of forms having a low-lying zero. We first prove that, along an infinite sequence of prime levels $N$, the number of such forms is $\gg N^{7/8}\log N$. We then use higher centered moments and an appropriate test function to show that a positive proportion of the family has a zero in a prescribed window. Finally, we obtain polynomial upper-tail bounds for the number of zeros occurring there and show that a positive proportion of forms have a bounded, nonzero number of low-lying zeros.
\end{abstract}

\maketitle

\vspace{-6mm}

\tableofcontents

\section{Introduction}
\subsection{Motivation}
Zeros of automorphic L-functions are closely connected to arithmetic questions, including the analytic ranks of elliptic curves. Their distribution near the central point is also expected to exhibit universal behavior governed by Random Matrix Theory \cite{BSD1, BSD2, katzSarnak1999book, katzSarnak1999zeros}. Assuming the Generalized Riemann Hypothesis (GRH), non-trivial zeros lie on the critical line, so we may write
\begin{equation}\rho_{f,j} = \frac{1}{2}+i\gamma_{f,j}.\end{equation}
Zeros at and near the central point are of particular interest because central zeros are related to arithmetic ranks, while the statistical behavior of low-lying zeros is predicted by Random Matrix Theory \cite{ILS, montgomery1973, rudnickSarnak1996}. To study the statistical behavior of zeros near the central point, Katz and Sarnak introduced the $n$-level density \cite{katzSarnak1999book, katzSarnak1999zeros}, which averages across members of a given family. The zeros are normalized by the analytic conductor $c_f$ so their mean local spacing is asymptotically one:
\begin{equation}\widetilde{\gamma}_{f,j} \ = \ \frac{\log c_f}{2\pi}\gamma_{f,j}.\end{equation}
Evaluating central zeros yields critical bounds on the proportion of $L$-functions with analytic rank $r \geq 2$ \cite{ILS, katzSarnak1999zeros}. The Katz--Sarnak conjecture predicts that the low-lying zeros in a family converge, after normalization, to the local statistics of eigenvalues near 1 in the corresponding classical compact group. For orthogonal families, this prediction in particular constrains the frequency of forms with high-order vanishing at the central point \cite{katzSarnak1999zeros}. Iwaniec, Luo, and Sarnak used the 1-level density to obtain quantitative bounds on the frequency of central zeros in families of cusp forms \cite{ILS}. Hughes and Miller subsequently tightened these theoretical limits by analyzing the centered moments of the one-level density, a method that circumvents the Fourier support restrictions of first-order statistics \cite{Hughes_2007}. Through the modularity of elliptic curves and the Birch--Swinnerton-Dyer conjecture, bounds on central vanishing in suitable families of modular $L$-functions have consequences for the distribution of elliptic curves of high rank. \cite{BSD1, BSD2}.

In this paper, we shift focus from individual bounds to examining the lowest zeros across an entire family. Previous studies have bounded this lowest zero for individual $L$-functions \cite{BSD1, BSD2, FM, hughesRudnick2003, M, Mes}. In this paper, we shift focus from individual bounds to examining the lowest zeros across an entire family of $L$-functions. We specifically study holomorphic cuspidal newforms of fixed even weight $k$ and prime level $N \to \infty$. By removing the known central zero contribution from odd forms, we focus on the even family, which is modeled by $SO(\mathrm{even})$ symmetry. The $n$-level densities for these families can be directly related to the centered moments of the one-level statistic \cite{cohen2026moments, Hughes_2007, hughesRudnick2003}:\begin{equation}D(f;\phi) \ = \ \sum_j\phi(\widetilde{\gamma}_{f,j}).\end{equation}
Building on methods from \cite{GoesMiller2010, Mes}, previous work gives an upper bound for the smallest normalized zero occurring among the forms in this family. Using a recent construction of \cite{ARORA2025262} of a specific even Schwartz function $\phi_\omega$, on obtains an explicit window $(-\omega,\omega)$ that contains a zero for at least one form in the family. Under current Fourier support limitations, this explicit calculation yields a window of approximately $0.21864$ times the average spacing.

We move beyond the existence of a low zero somewhere in the family and ask how many forms contain a zero in a prescribed window about the central point. We obtain lower bounds for the number of forms containing at least one zero in the window, prove that a positive proportion of forms contain such a zero for suitable windows, and control the possible concentration of these zeros on individual forms. In this section, we provide the necessary background for the counting arguments developed in this paper. First, we introduce the family of modular forms and the associated $L$-functions, followed by the normalized zero statistics and centered moments used to detect zeros near the central point. Finally, we describe the test functions and previously known first-zero bounds. Notations and conventions have been listed explicitly in Section \ref{notation}. 

\subsection{Forms and $L$-Functions}
Recall that a matrix $\ga\in \SL_2(\Z)$ of the form
\begin{equation}
\ga \ = \ \begin{bmatrix}
a & b\\
c & d
\end{bmatrix}
\end{equation}
acts on a complex number $z$ by
\begin{equation}
\ga z \ = \ \f{az+b}{cz+d}.
\end{equation}
Functions on the upper half-plane that behave in a certain way under the action of this group are studied extensively in analytic number theory.

\begin{definition}
A modular form of weight $k$ is a holomorphic function $f:\HH\to\C$ that satisfies
\begin{equation}
f(\ga z) \ = \ (cz+d)^k f(z)
\end{equation}
for all
\begin{equation}
\ga \ = \ \begin{bmatrix}
a & b\\
c & d
\end{bmatrix}\in \operatorname{SL}_2(\Z).
\end{equation}
We also require that $f$ be bounded as $\Im(z)\to\infty$.
\end{definition}

We can impose some additional conditions on modular forms. When a modular form satisfies $f(z)\to 0$ as $\Im(z)\to \infty$, we say that $f$ is also a cusp form.

\begin{definition}
A modular form of weight $k$ and \textbf{level $N$} is a function that must satisfy the transformation rule when $c\equiv 0\pmod N$ for $\ga \ = \ \begin{bsmallmatrix}
a & b\\
c & d
\end{bsmallmatrix}$. We also require the form to be holomorphic at every cusp.

We say that $f$ is a newform of level $N$ if it does not arise from a form of lower level $M<N$, where $M\mid N$. We use $H^\star_k(N)$ to denote the set of normalized holomorphic cuspidal newforms of weight $k$ and level $N$.
\end{definition}

For each cuspidal newform, we associate an \textit{$L$-function}.

\begin{definition}
For every $f\in H^\star_k(N)$, consider the Fourier expansion
\begin{equation}
f(z) \ = \ \sum_{n=1}^\infty a_f(n)e^{2\pi i n z}.
\end{equation}
Let
\begin{equation}
\lm_f(n) \ = \ a_f(n)n^{-(k-1)/2}
\end{equation}
and define the $L$-function associated to $f$ as
\begin{equation}
L(s,f) \ = \ \sum_{n=1}^\infty \lm_f(n)n^{-s}.
\end{equation}
Its completion is given by
\begin{equation}
\Lambda(s,f) \ := \ \left(\f{\sqrt{N}}{2\pi}\right)^s
\Gamma\left(s+\f{k-1}{2}\right)L(s,f).
\end{equation}
\end{definition}

Then we split $H^\star_k(N)$ into two subsets depending on the functional equation of $\Lambda(s,f)$.

\begin{definition}
Note that $\Lambda(s,f)$ satisfies
\begin{equation}
\Lambda(s,f) \ = \ \e_f\Lambda(1-s,f),
\end{equation}
where $\e_f=\pm1$. We define
\begin{align}
H^+_k(N) &\ := \ \Set{f\in H^\star_k(N)\mid \e_f=1},\\
H^-_k(N) &\ := \ \Set{f\in H^\star_k(N)\mid \e_f=-1},
\end{align}
which are called the sets of even forms and odd forms, respectively. 
Here, 
\begin{equation}
    \Ff_N \ := \ H^+_k(N).
\end{equation}
\end{definition}

\subsection{Low-Lying Zeros and $n$-Level Densities}
The zeros of $L$-functions are extensively studied in the literature, especially the zeros closer to the central point, as they are of interest in number-theoretic settings and tied to valuable arithmetic information. Two complementary statistical regimes are particularly relevant here. The $n$-level correlation statistics describe local correlations among zeros high on the critical line of a fixed $L$-function, whereas the Katz--Sarnak framework studies zeros near the central point across a family as the analytic conductor tends to infinity \cite{katzSarnak1999book, katzSarnak1999zeros}.

The study of correlation statistics is rooted in Montgomery's 1972 discovery of a profound connection between the distribution of the Riemann zeta function's zeros and the eigenvalue distribution of random Hermitian matrices \cite{montgomery1972}. He conjectured that the $n$-level correlation of these zeros exactly matches that of the Gaussian Unitary Ensemble (GUE). This conjecture was later verified for suitably restricted test functions by Hejhal \cite{Hej} for the 3-level correlation of the zeta function, and by Rudnick and Sarnak \cite{RS1} for all $n$ across all automorphic $L$-functions. The collection of $n$-level correlation statistics determines the limiting local point-process statistics, and hence encodes the distribution of spacings between nearby zeros. This theoretical spacing was famously corroborated by Odlyzko's numerical computations near the $10^{20}$-th zero, which observed that the main term of the spacing between nearby zeros closely aligned with GUE predictions (see \cite{odlyzko1987, H, BFMT-B, Con} and their references for more details on these connections). 

Random Matrix Theory predicts that, after normalization, the local statistics of low-lying zeros agree with those of eigenvalues near 1 in the corresponding classical compact group.. To study the zeros near the central point, Katz and Sarnak introduced the $n$-level density, which is affected almost exclusively by contributions of zeros near to the central point relative to the average spacing there.

First, we note that the zeros near the central point are, on average, spaced apart by a distance comparable to
\begin{equation}
\f{2\pi}{\log(c_f)},
\end{equation}
where $c_f$ is called the \textit{analytic conductor}. For fixed weight $k$ and prime level $N$, the analytic conductor is independent of the choice of $f$ in the even cuspidal newforms family and is proportional to $kN^2$. Using this as a scaling factor, we define the $n$-level density.

\begin{definition}[$n$-level Density]
Let $\Phi:\R^n\to\R$ be a test function. The $n$-level density of an $L$-function $L(s,f)$ is given by
\begin{equation}
D_n(f;\Phi)
\ :=\
\sum_{\substack{j_1,\ldots,j_n\\
j_r\neq\pm j_s\text{ for }r\neq s}}
\Phi\left(
\f{\log c_f}{2\pi}\ga_{f,j_1},
\ldots,
\f{\log c_f}{2\pi}\ga_{f,j_n}
\right),
\end{equation}
where the nontrivial zeros (assuming GRH, so all zeros lie on the critical line) of $L(s,f)$ are written as
\begin{equation}
\rho_{f,j} \ = \ \f{1}{2}+i\ga_{f,j}.
\end{equation}
\end{definition}

Note that we are summing over $n$-tuples of zeros in which no two entries of the tuple are equal up to sign. The test function is chosen to optimize both the bounds and the feasibility of calculations. We see behavior similar to the Heisenberg uncertainty principle: if the test function is concentrated near the central point and decays rapidly, we can gain more information about the zeros closer to the central point. However, as we approximate a delta spike, the Fourier transform becomes more spread out, leading to more terms in the number-theoretic expansions until the available averaging formulas no longer apply. We optimize over this subset of test functions based on additional criteria.

\begin{remark}
All zero locations are measured in units of the mean spacing near the central point. Thus a normalized distance of 1 corresponds to an unnormalized distance of approximately $2\pi/\log(c_f)$.
\end{remark}

In the following sections, $\phi$ is an even Schwartz function whose Fourier transform $\hat{\phi}$ has compact support.

\subsection{Centered Moments}
In certain cases, the $n$-th centered moments are used instead of the $n$-level density. These are essentially equivalent after some combinatorics. However, the expansions of the $n$-th centered moments are better suited for obtaining bounds, as shown by Hughes--Miller \cite{Hughes_2007}.

\begin{definition}
The $n$-th centered moment for the families of even or odd cuspidal newforms of level $N$ is defined by
\begin{equation}
\mathcal{D}_n^{\pm}(N,\phi)
\ :=\
\la
\left(D(f;\phi)-\la D(f;\phi)\ra_{N;\pm}\right)^n
\ra_{N;\pm},
\end{equation}
where $D(f;\phi)$ is the one-level density of $f$. Here, $\la\cdot\ra_{N;\pm}$ denotes the average over the family of even or odd level-$N$ cuspidal newforms, where $\pm$ denotes the even or odd family.
\end{definition}

Centering removes the deterministic first-order contribution and isolates the fluctuation of the one-level statistic from form to form. This is particularly useful because a sign-changing test function produces a common baseline for forms with no zeros in a specified window. The centered moments then measure how far the family average is forced to be from that baseline.

We have the following theorem used to calculate the $n$-th centered moment, given by Theorem 1.2 of \cite{cohen2026moments}. 

\begin{theorem}\label{densityeqn}
Let $n\geq 2$ and let $a$ be an integer that satisfies
\begin{equation}
0\leq a\leq \f{n}{2}.
\end{equation}
(In specific restricted expansions, one may substitute the parameter configuration \begin{equation} a \ = \ \f{n}{2} \end{equation} depending on the summation definitions.) 
Suppose that
\begin{equation}
\supp(\hat{\phi})
\ \subset\
\left(-\f{1}{n-a},\f{1}{n-a}\right).
\end{equation}
Assume GRH for Dirichlet $L$-functions and for cuspidal newforms and their symmetric squares. Then
\begin{equation}
\lim_{\substack{N\to\infty\\N\text{\rm prime}}}
\mathcal{D}_n^{\pm}(N,\phi)
\ =\
\mathbbm{1}_{\mathrm{even}}(n)(n-1)!!
\left(V_\phi^2\right)^{n/2}
\ \pm\ S(n,a;\phi),
\end{equation}
where
\begin{equation}
V_\phi^2
\ :=\
2\int_{-\infty}^\infty
|y|\hat{\phi}(y)^2\,dy
\end{equation}
and
\begin{equation}
S(n,a;\phi)
\ :=\
\sum_{\ell=0}^{\left\lfloor \f{a-1}{2}\right\rfloor}
\f{n!}{(n-2\ell)!\ell!}
\mathcal{R}(n-2\ell,a-2\ell;\phi)
\left(\f{V_\phi^2}{2}\right)^\ell,
\end{equation}
with
\begin{equation}\label{eq: R def}
\begin{split}
\mathcal{R}(m,i;\phi)
\ :=\ &
2^{m-1}(-1)^{m+1}
\sum_{\ell=0}^{i-1}
(-1)^\ell\binom{m}{\ell}\\
&\times
\left(
-\f{1}{2}\phi(0)^m
+
\int_{-\infty}^\infty\cdots
\int_{-\infty}^\infty
\hat{\phi}(x_2)\cdots\hat{\phi}(x_{\ell+1})\right.\\
&\hspace{1.3in}\left.
\times
\int_{-\infty}^\infty
\phi(x_1)^{m-\ell}
\f{
\sin\left(
2\pi x_1
(1+|x_2|+\cdots+|x_{\ell+1}|)
\right)
}{2\pi x_1}\,
dx_1\cdots dx_{\ell+1}
\right).
\end{split}
\end{equation}
\end{theorem}

The quantity $S(n,a;\phi)$ is the lower-order contribution in the asymptotic centered moment and is the term that will be used to detect low-lying zeros.
It is convenient to parametrize the support by
\begin{equation}
\si \ = \ \f{n}{n-a},
\end{equation}
Then the support condition may be written as
\begin{equation}
\supp(\hat{\phi})
\ \subset\
\left(-\f{\si}{n},\f{\si}{n}\right),
\end{equation}
where $1\leq\si\leq 2$ with $\si<2$ when $n$ is odd. In this case, we may write $S(n,\si;\phi)$ to mean $S(n,a;\phi)$ since $a$ will be chosen based on $\si$.

\begin{remark}
    In the case that $n=1$, we write $S(1,\si;\phi)$ to mean the 1-level density. Explicit formulas for $S(1,\si;\phi)$ can be found in \cite{ILS} and \cite{ARORA2025262}.
\end{remark}

\subsection{Test Functions for Low-Lying-Zero Bounds}
In our applications, the test function $\phi$ is required to have Fourier transform supported within the range for which the relevant density or centered-moment formula is available. The sign of the test function is chosen according to the desired inequality. Nonnegative test functions are useful for upper bounds on zero counts, whereas sign-changing test functions with a prescribed sign pattern relative to the target interval are used to prove existence and frequency results.

To obtain upper bounds on zero counts, we use non-negative test functions ($\phi(x) \ge 0$), dropping non-negative contributions from zeros outside the target interval. While past work \cite{BCDMZ, FM, ILS} constructs optimal non-negative functions, the quantitative gain is minor, so we employ the naive test function for simplicity.

\begin{definition}[Naive Test Function]
Let $\si_n$ denote the maximum available support. The naive test function for the $n$-level density is defined by
\begin{equation}
\phi_{\mathrm{naive};n}(x) \ := \ \left(\frac{\sin(\pi\si_n x)}{\pi\si_n x}\right)^2,
\end{equation}
with Fourier transform
\begin{equation}
\hat{\phi}_{\mathrm{naive};n}(y) \ = \ \frac{1}{\si_n}\left(1 - \frac{|y|}{\si_n}\right)
\end{equation}
for $|y| < \si_n$, and $\hat{\phi}_{\mathrm{naive};n}(y) = 0$ otherwise.
\end{definition}

Quantitative lower bounds and existence proofs require a sign-changing test function $\phi$ satisfying $\phi(x) \ge 0$ for $|x| \le \om$, $\phi(x) \le 0$ for $|x| > \om$, and $\phi(0) \neq 0$. This profile allows dropping non-positive contributions from zeros outside $[-\om,\om]$. In practice, such functions are constructed by first specifying a compactly supported Fourier transform $\hat{\phi}$ in frequency space.

\begin{definition}
Let $h(y)$ be an even function that is at least twice continuously differentiable, supported in $[-1,1]$, and monotonically decreasing on $[0,1]$. Let
\begin{equation}
f(y)\ :=\ h(2yn/\si)
\end{equation}
and let
\begin{equation}
g(y)\ :=\ (f*f)(y)
\end{equation}
be the convolution of $f$ with itself. We define
\begin{equation}
\phi_\om(x)
\ =\
\left(1-(x/\om)^2\right)\hat{g}(x).
\end{equation}
\end{definition}

The factor $1-(x/\om)^2$ forces $\phi_\om$ to vanish at $x=\pm\om$. Moreover, its Fourier transform satisfies
\begin{equation}
\hat{\phi}_\om(y)
\ =\
g(y)+(2\pi\om)^{-2}g''(y),
\end{equation}
and
\begin{equation}
\supp(\hat{\phi}_\om)
\ \subset\
\left(-\f{\si}{n},\f{\si}{n}\right).
\end{equation}
Also, $\phi_\om(x)$ is nonnegative when $|x|\leq\om$ and nonpositive when $|x|>\om$.

\subsection{Existing Bounds for Low-Lying Zeros}
Using the centered-moment theorem above, the authors of \cite{ARORA2025262} prove that there is at least one form with at least one zero below the bound $\om$. We restate these results here.

\begin{theorem}\label{result1}
Assume GRH. For a family consisting of even forms of prime level and for an odd $n$, if $\om$ satisfies
\begin{equation}
-\left(
\widehat{\phi}_{\omega}(0)
+
\f{1}{2}
\int_{-\si/n}^{\si/n}
\widehat{\phi}_{\omega}(y)\,dy
\right)^n
\ <\
S(n,a;\phi_{\omega}),
\end{equation}
then there exists at least one form with at least one normalized zero in the interval $(-\omega, \omega)$.
\end{theorem}

\begin{theorem}\label{result2}
 For a family consisting of even forms of prime level and for an odd $n$, the first normalized zero of the family of cuspidal newforms is in the interval $(-\omega_{\min}, \omega_{\min})$ where
\begin{equation}
\omega_{\min}(\sigma, h)
\ >\
\left(
\f{
-\si\displaystyle\int_0^1 h(u)^2\,du
+
\f{\si^2}{4}
\displaystyle\int_{-1}^1
\int_{0}^{2/\si}
h(u)h(v-u)\,dv\,du
}{
\f{1}{\si}
\displaystyle\int_0^1 h(u)h''(u)\,du
+
\f{1}{4}
\displaystyle\int_{-1}^1
\int_{0}^{2/\si}
h(u)h''(v-u)\,dv\,du
}
\right)^{-1/2}
\pi^{-1}.
\end{equation}
\end{theorem}

Note that in \cite{ARORA2025262}, they find the explicit bound $\om_{\text{min}} \approx 0.21864$ for $h = \cos(\pi y/2)$. These results establish the existence of low-lying zeros but do not measure how broadly the zeros are distributed through the family. Thus, the remainder of the paper is dedicated to keeping track of the forms responsible for the moment excess and combining this information with the bounds for the number of zeros that a single form can contribute.

\subsection{Results}
\label{results}

The theorems from \cite{ARORA2025262} do not distinguish whether there are more zeros in the given window and how the zeros are distributed across forms in the family. To address this, we introduce the following definitions.

\begin{restatable}{definition}{aNbNdefn}
    Let $a_{\om, N}$ denote the number of forms in $\Ff_N$ with at least one normalized zero in $(-\om,\om)$, and let $\widetilde{\Ff}_N$ be the collection of such forms. Let $b_{\om, N}$ be the maximum number of zeros in $(-\om,\om)$ on any single form.
\end{restatable}

Using the centered-moment argument to control the product $a_{\om,N}b_{\om,N}$ and bounding $b_{\om, N}$, we obtain

\begin{restatable}{theorem}{asymptoticthm}\label{asymptoticthm}
Assume GRH.  If $\om$ satisfies

\begin{equation}
        -\left(\hat{\phi}(0)+\f{1}{2}\int_{-\infty}^\infty\hat{\phi}(y)dy\right) \ < \  S(1, \si; \phi_{\omega}),
    \end{equation}
    then 
    \begin{equation}
        \limsup_{\substack{N\to\infty\\N \text{\rm prime}}}\f{a_{\om,N}}{N^{7/8} \log N} \ >\ 0.
\end{equation}
\end{restatable}

\begin{restatable}{remark}{explicitbd}
    In particular, for $h = \cos(\pi y/2)$, the bound in \ref{asymptoticthm} is $\om_{\text{min}} \approx 0.21864$. This is computed in \cite{ARORA2025262}.
\end{restatable}

To establish a positive proportion, we make the following definition.

\begin{restatable}{definition}{Rfdefn}
     Let
\begin{equation}
R_f(\om) \ = \ \#\Set{j:\abs{\widetilde{\ga}_{f,j}}\leq\om},
\end{equation}
where the zeros are counted with multiplicity.
\end{restatable}

By constructing an even Schwartz function $\varphi$ bounded below on $[-\om,\om]$ with sufficiently small Fourier support, we find
\begin{equation}
R_f(\om) \ \ll_{\varphi} \ D(f;\varphi).
\end{equation}

When the corresponding odd centered-moment inequality is strict, we obtain

\begin{restatable}{theorem}{positivepercent}\label{positivepercent}
Assume GRH. Let $n\geq3$ be odd, and suppose that $\om$ satisfies the condition $(-\mu)^n < S(n,a;\phi_{\om})$, where $\mu := \mu(\phi_{\om}, \Ff)$. Then
\begin{equation}
\liminf_{\substack{N \to \infty \\ N \text{\rm prime}}}\f{a_{\om, N}}{\abs{\Ff_N}} \ > \ 0.
\end{equation}
In other words, a positive proportion of forms in $\Ff_N$ have at least one normalized zero in $(-\om, \om)$.
\end{restatable}

\begin{restatable}{corollary}{linearnumberofforms}\label{linearnumberofforms}
Under the same assumptions as Theorem \ref{positivepercent} we have that
\begin{equation}
a_{\om, N} \ \gg_{\om, n} \ \abs{\Ff_N}.
\end{equation}
In particular, using that $\abs{\Ff_N}\asymp_k N$ for fixed $k$ and prime $N$, we have $a_{\om,N}\gg_{\om,n,k}N$.
\end{restatable}

\begin{restatable}{corollary}{zerocountsubseq}\label{zerocountsubseq}
Under Theorem \ref{positivepercent}'s assumptions, there exists an integer $r \geq 1$ such that
\begin{equation}
\limsup_{\substack{N \to \infty \\ N \text{\rm prime}}} \f{\abs{\Set{f \in \Ff_N\mid R_f(\om) \ = \ r}}}{\abs{\Ff_N}} \ > \ 0.
\end{equation}
\end{restatable}

\section{Asymptotics on Number of Zeros}

We follow the notation of \cite{ARORA2025262} and the proof of their Theorem 1.6. We begin with the preliminary lemma.
\begin{lemma}\label{mean lemma}
    The mean of the family, $\Ff_N$, over $SO(\text{\rm even})$ is given by
    \begin{equation}
        \mu(\phi,\Ff) \ :=\ \hat{\phi}(0)+\f{1}{2}\int_{-\infty}^\infty\hat{\phi}(y)dy.
    \end{equation}
\end{lemma}
\begin{proof}
    See \cite{ILS}, \cite{Hughes_2007}, and \cite{ARORA2025262}.
\end{proof}

Let $a_{\om,N}$ denote the number of forms $f\in \Ff_N$ with at least one normalized zero below $\om$. Let us denote these forms by $\widetilde{\Ff_N}$. Among these forms, let $b_{\om,N}$ denote the maximum number of normalized zeros below $\om$ for any form $f\in \widetilde{\Ff}_N$.

\begin{theorem}
    Assume GRH. For an odd $n$, if $\om$ satisfies
     \begin{equation}
        -\left(\hat{\phi}(0)+\f{1}{2}\int_{-\infty}^\infty\hat{\phi}(y)dy\right)^n \ < \ S(n, \si; \phi_{\omega}),
    \end{equation}
    then
    \begin{align}
        \lim_{N\to \infty}\f{a_{\om,N}}{|\Ff_N|}\sum_{i=1}^n\binom{n}{i}(Bb_{\om,N})^i(-\mu)^{n-i}
    \end{align}    
    cannot go to 0 as $N$ tends to infinity through the primes.
\end{theorem}
\begin{proof}
    First, note that $\phi_\om$ is non-positive when $|x|> \om$. Since $n$ is odd, we see that
    \begin{equation}
        \lim_{\substack{N\to\infty \\ N \text{\rm prime}}} \frac{1}{|\mathcal{F}_{N}|}\sum_{f \in \mathcal{F}_{N}} \left( \sum_{|\gamma_{f, j}| \leq \omega}\phi_{\omega}(\widetilde{\gamma}_{f, j}) -  \mu(\phi_{\omega}, \mathcal{F}) \right)^{n} \ \geq \ \lim_{\substack{N\to\infty \\ N \text{\rm prime}}} \frac{1}{|\mathcal{F}_{N}|}\sum_{f \in \mathcal{F}_{N}} \left( \sum_{j}\phi_{\omega}(\widetilde{\gamma}_{f, j}) -  \mu(\phi_{\omega}, \mathcal{F}) \right)^{n}.
    \end{equation}
    Thus, 
    \begin{equation}
        \lim_{\substack{N\to\infty \\ N \text{\rm prime}}} \frac{1}{|\mathcal{F}_{N}|}\sum_{f \in \mathcal{F}_{N}} \left( \sum_{|\gamma_{f, j}| \leq \omega}\phi_{\omega}(\widetilde{\gamma}_{f, j}) -  \mu(\phi_{\omega}, \mathcal{F}) \right)^{n} \ \geq \  S(n, \si; \phi_{\omega}).
    \end{equation}
    Now, by our definition of $a_{\om,N}$ and $b_{\om,N}$, we see that the expression within the limit is given by
    \begin{equation}
        -\mu^n\f{|\Ff_N|-a_{\om,N}}{|\Ff_N|} + \f{1}{|\Ff_N|} \sum_{f\in \widetilde{\Ff_N}} \left( \sum_{|\ga_{f,j}|<\om} \phi_{\om}(\widetilde{\ga}_{f,j}) - \mu \right)^n.
    \end{equation}
    We can rewrite this as 
    \begin{equation}
        -\mu^n + \f{1}{|\Ff_N|} \left(\sum_{f\in \widetilde{\Ff_N}} \left( \sum_{|\ga_{f,j}|<\om} \phi_{\om}(\widetilde{\ga}_{f,j}) - \mu \right)^n +a_{\om,N}\mu^n\right) .
    \end{equation} 
    Let $B$ bound $\phi_\om$. Then, we have
    \begin{equation}
        -\mu^n + \f{1}{|\Ff_N|} \left(\sum_{f\in \widetilde{\Ff_N}} \left( \sum_{|\ga_{f,j}|<\om} \phi_{\om}(\widetilde{\ga}_{f,j}) - \mu \right)^n +a_{\om,N}\mu^n\right) \ \leq \ -\mu^n  +\f{a_{\om,N}\left( (Bb_{\om,N} - \mu) ^n+\mu^n\right)}{|\Ff_N|}.
    \end{equation} 
    Since $n$ is odd, this simplifies to 
    \begin{equation}
        -\mu^n+\f{a_{\om,N}}{|\Ff_N|}\sum_{i=1}^n\binom{n}{i}(Bb_{\om,N})^i(-\mu)^{n-i}.
    \end{equation}
    Putting these equations together, we have
    \begin{equation}
        \lim_{\substack{N\to\infty \\ N \text{\rm prime}}}\left(-\mu^n+\f{a_{\om,N}}{|\Ff_N|}\sum_{i=1}^n\binom{n}{i}(Bb_{\om,N})^i(-\mu)^{n-i}\right) \ \geq \ S(n,\si;\phi_\om).
    \end{equation}
    If 
    \begin{equation}
        \lim_{\substack{N\to\infty \\ N \text{\rm prime}}}\left(\f{a_{\om,N}}{|\Ff_N|}\sum_{i=1}^n\binom{n}{i}(Bb_{\om,N})^i(-\mu)^{n-i}\right) \ = \ 0,
    \end{equation}
    then 
    \begin{equation}
        -\mu^n \ \geq \ S(n,\si;\phi_\om).
    \end{equation}
    However, by Lemma \ref{mean lemma}, this implies that
    \begin{equation}
        -\left(\hat{\phi}(0)+\f{1}{2}\int_{-\infty}^\infty\hat{\phi}(y)dy\right)^n \ \geq \ S(n,\si;\phi_\om)
    \end{equation}
    yielding a contradiction.
\end{proof}

\begin{corollary}
    Assume GRH. If $\om$ satisfies
    \begin{equation}
        -\left(\hat{\phi}(0)+\f{1}{2}\int_{-\infty}^\infty\hat{\phi}(y)dy\right) \ < \ S(1, \si; \phi_{\omega}),
    \end{equation}
    then as $N$ tends to $\infty$ through the primes,
    \begin{align}
        \f{a_{\om,N}b_{\om,N}}{|\Ff_N|}
    \end{align}    
    cannot go to 0. 
\end{corollary}

Now, let us find explicit bounds on $a_{\om,N}$.
We restate

\asymptoticthm*

\begin{proof}
    Let $\varphi$ be a test function such that $\varphi \geq 0$ and $\varphi \geq 1$ on $[-\omega,\omega]$. We also need $\hat{\varphi}$ to have compact support in $[-1/2,1/2]$.

    To verify that this is indeed possible, we give the following example:
    \begin{equation}
        \varphi(x)  \ = \  \left( \frac{\sin\left(\frac{\pi \sigma x}{4}\right)}{\frac{\pi \sigma x}{4}} \right)^4 \left( \frac{\frac{\pi  \omega}{4}}{\sin\left(\frac{\pi  \omega}{4}\right)} \right)^4
    \end{equation}
    where we can adjust $\om$ slightly to avoid dividing by 0.
    Since $|\Ff_N|$ is finite, the maximum number of zeros below $\om$ is attained by some form, $f_0$.
    We see that for such a form $f_0$, we have
    \begin{equation}
        b_{\om,N} \leq \sum_{\gamma \leq \omega}\varphi(\gamma) \leq \sum_\gamma \varphi(\gamma)  \ = \  D(f_0,\varphi).
    \end{equation}
    where the sum is taken over the nontrivial zeros of the $L$-function associated to $f_0$, counted with multiplicity.
    
    Now, from \cite{ILS}, we have
    \begin{align}
    D(f;\varphi) & \ = \  \int_{-\infty}^\infty \varphi \, dx \nonumber\\
    &\quad + \frac{2}{\log N} \sum_{p \neq N} \frac{\log p}{p} \hat{\varphi}\left(\frac{2 \log p}{\log N}\right) \nonumber\\
    &\quad - \frac{2}{\log N} \sum_{p \neq N} \frac{\lambda_f (p^2)\log p}{p} \hat{\varphi}\left(\frac{2 \log p}{\log N}\right) \nonumber\\
    &\quad - \frac{2}{\log N} \sum_{p \neq N} \frac{\lambda_f (p)\log p}{\sqrt{p}} \hat{\varphi}\left(\frac{2 \log p}{\log N}\right) + O\left(\frac{1}{\log N}\right).
    \end{align}

    Since $\hat{\varphi}$ has bounded support, the only non-trivial terms in the sum occur when $\log p \leq \frac{1}{4} \log N$. In fact, the support is contained in $[-1/2,1/2]$. Thus, we need $p \leq N^{1/4}$.  Let $B$ bound $\varphi$. We have
    \begin{align}
    D(f;\varphi) & \ \leq \ \int_{-\infty}^\infty \varphi \, dx \notag \\
    &\quad + \frac{2B}{\log N} \sum_{p \leq N^{1/4}} \frac{\log p}{p} \notag \\
    &\quad + \frac{2B}{\log N} \sum_{p \leq N^{1/4}} \frac{|\lambda_f (p^2)|\log p}{p} \notag \\
    &\quad + \frac{2B}{\log N} \sum_{p \leq N^{1/4}} \frac{|\lambda_f (p)|\log p}{\sqrt{p}}  + O\left(\frac{1}{\log N}\right).
    \end{align}
    It is known from Deligne's bound that $\abs{\lambda_f(p)} \ \leq \ 2$ and $\abs{\lambda_f(p^2)} \ \leq \ 3$; see \cite{iwaniec2004analytic}.
    Now, the first term of the density is a constant. We also know that
    \begin{align}
        \sum_{p \leq x}\f{\log p}{p}   \ = \ O(\log x) ,
    \end{align}
    and 
    \begin{align}
         \sum_{p \leq x} \f{\log p}{\sqrt{p}} \ = \  O(\sqrt{x}) .
    \end{align}
    Therefore,
    \begin{equation}
    D(f;\varphi) \ \leq \ C + O(1) + O(1) + O\left(\f{\sqrt{N^{1/4}}}{\log N}\right) + O\left(\frac{1}{\log N}\right).
    \end{equation}
    Then
    \begin{equation}
        b_{\om,N} \ \leq \ D(f; \varphi) \ \leq \ O\left( \f{N^{1/8}}{\log N}\right).
    \end{equation}
    Finally, since $|\Ff_N|  \ \asymp \ N$ and
    \begin{equation}
        \f{a_{\om,N}b_{\om,N}}{|\Ff_N|} \ \not\to \ 0,
    \end{equation}
    we must have
    \begin{equation}
        \limsup_{\substack{N\to\infty\\N \text{\rm prime}}}\f{a_{\om,N}}{N^{7/8} \log N} \ >\ 0.
    \end{equation}
\end{proof}

\explicitbd*
\section{Bounds on the Number of Zeros}\label{bounds}
The arguments in the previous subsection find a pointwise bound on the maximum number of zeros that a single form can have in a given window. We now move past this and take a different approach; rather than bounding every form individually, we use a second test function to bound a high moment of the number of zeros on average. This then lets us apply Cauchy--Schwarz to prove that a positive proportion of forms have a zero in our window $[-\om, \om]$. We begin by defining for $f \in \Ff_N$,
\begin{equation}
R_f(\om) \ :=\ \#\Set{j\mid \abs{\widetilde{\ga}_{f,j}}\leq\om}, \quad \text{and} \quad X_f \ :=\ \sum_{\abs{\widetilde{\ga}_{f,j}}\leq\om}\phi_{\om}(\widetilde{\ga}_{f,j})
\end{equation}
where the zeros are counted with multiplicity. Additionally, we write $\mu := \mu(\phi_{\om}, \Ff)$.

\subsection{An Auxiliary Test Function}
For the averaging argument, we need a test function with a different role from $\phi_{\om}$. This test function should be nonnegative everywhere and bounded below on the chosen window such that its one-level statistics dominate the actual number of zeros there. At the same time, its Fourier support must be small enough to actually apply the centered-moment theorem at order $2n$. Standard construction techniques are given in \cite{GoesMiller2010, hughesRudnick2003, ILS, ARORA2025262}.

\begin{lemma}\label{auxtestfunc}
Let $n$ be a positive integer and let $\om \ > 0 $. Then there exists a nonnegative even Schwartz function $\varphi$ such that
\begin{equation}
c_{\varphi} \ := \ \min_{\abs{x} \leq \om}\varphi(x) \ > \ 0
\end{equation}
and
\begin{equation}
\operatorname{supp}\left(\widehat{\varphi}\right) \ \subset \ \biggl(-\frac{1}{n}, \frac{1}{n}\biggr).
\end{equation}
\end{lemma}

\begin{proof}
Choose a nonzero real-valued even nonnegative function $\eta \in C_c^\infty(\R)$ supported in $(-\frac{1}{4n},\frac{1}{4n})$, and let $q$ be its inverse Fourier transform. Since $\eta$ is real-valued and even, $q$ is real-valued and even. We have
\begin{equation}
q(0)\ =\ \int_{-\infty}^{\infty}\eta(y)dy>0.
\end{equation}
Thus, since $q$ is continuous there exists some $\delta > 0$ such that $q(x) > 0$ whenever $\abs{x} \leq \delta$. Now, choose $T > \max\{\om/\delta, 1\}$, and define
\begin{equation}
q_T(x) \ := \ q(x/T) \qquad \text{and} \qquad \varphi(x) \ := \ q_T(x)^2.
\end{equation}
The function $\varphi$ is nonnegative, even, and Schwartz. Thus, if $\abs{x} \leq \om$, then $\abs{x/T} < \delta$. Hence, $q_T(x) > 0$, which implies $c_{\varphi} > 0$. All that remains to do is to check the support, and by scaling, we can see that $\widehat{q_T}$ is supported on $(-\frac{1}{4nT},\frac{1}{4nT})$. Thus, since the Fourier transform of a product is a convolution, we have that
\begin{equation}
\widehat{\varphi} \ = \ \widehat{q_T}*\widehat{q_T},
\end{equation}
yielding
\begin{equation}
\supp(\widehat{\varphi})\ \subset\ \left(-\f{1}{2nT},\f{1}{2nT}\right)\ \subset\ \left(-\f{1}{n},\f{1}{n}\right).
\end{equation}
\end{proof}

\subsection{A High-Moment Bound For The Zero Count}
\begin{lemma}\label{highmomentbound}
Let $n$ be a positive integer and let $\varphi$ be the function from Lemma \ref{auxtestfunc}. Then there exists a constant $C_{\varphi, n, \om} > 0$ that is independent of $N$ such that
\begin{equation}
\f{1}{\abs{\Ff_N}}\sum_{f\in\Ff_N}R_f(\om)^{2n} \ \leq \ C_{\varphi, n, \om}
\end{equation}
for all sufficiently large prime $N$.
\end{lemma}

\begin{proof}
Since $\varphi \geq c_{\varphi}$ on $[-\om, \om]$, every zero in this window contributes at least $c_{\varphi}$ to $D(f;\varphi)$. Therefore, we have that
\begin{equation}
c_{\varphi}R_f(\om) \ \leq \ D(f;\varphi)
\end{equation}
and further 
\begin{equation}\label{densitycontrolledzeros}
R_f(\om)^{2n} \ \leq \ c_{\varphi}^{-2n}D(f;\varphi)^{2n}.
\end{equation}
Now, an application of the centered-moment theorem at moment order $2n$, with its parameter set equal to $n$, gives the support condition
\begin{equation}
\supp(\widehat{\varphi}) \ \subset \ \left(-\f{1}{2n-n},\f{1}{2n-n}\right) \ = \ \left(-\f{1}{n},\f{1}{n}\right)
\end{equation}
which is exactly the support we can obtain from earlier with Lemma \ref{auxtestfunc}. Thus, we can say that the $2n$-th centered moment is bounded, consequently giving
\begin{equation}
\f{1}{\abs{\Ff_N}}\sum_{f\in\Ff_N}\abs{D(f;\varphi)-\la D(f;\varphi)\ra_{N;+}}^{2n} \ \ll_{\varphi,n} \ 1.
\end{equation}
Therefore, our one-level density formula \cite{barrett2017} yields $\la D(f;\varphi)\ra_{N;+}\to\mu(\varphi,\Ff)$, implying that our means are also uniformly bounded. Hence, if we use the fact that
\begin{equation}
D(f;\varphi)^{2n} \ \leq \ 2^{2n-1}\left(\abs{D(f;\varphi)-\la D(f;\varphi)\ra_{N;+}}^{2n} + \abs{\la D(f;\varphi)\ra_{N;+}}^{2n}\right)
\end{equation}
we can write that
\begin{equation}
\f{1}{\abs{\Ff_N}}\sum_{f\in\Ff_N}D(f;\varphi)^{2n} \ \ll_{\varphi,n} \ 1.
\end{equation}
Combining the preceding estimate with equation~\eqref{densitycontrolledzeros} gives us the desired bound.
\end{proof}

This lemma provides the uniform $2n$-th moment bound needed in the averaging argument. The odd moment will give a positive average supported on forms with a zero in the window, while the current even-moment estimate tells us that this average cannot be produced by an asymptotically negligible set of forms.

\subsection{A Positive Proportion of Forms} 
We proceed by using the estimates we have obtained previously to show that a positive proportion of the family has a zero in the interval $(-\om,\om)$.

\positivepercent*

\begin{proof}
For each $f \in \Ff_N$ we define
\begin{equation}
\Pi_f \ := \ (X_f-\mu)^n-(-\mu)^n.
\end{equation}
We note that $\Pi_f$ is nonnegative. In fact, $\phi_{\om}$ is nonnegative inside the window, hence, $X_f \geq 0$. Then since $n$ is odd, we have that the function $x \mapsto x^n$ is increasing, so
\begin{equation}
(X_f-\mu)^n \ \geq \ (-\mu)^n.
\end{equation}
Additionally, if $f$ has no normalized zeros in the window $(-\om, \om)$ then $X_f = 0$ and $\Pi_f = 0$. Therefore, the only forms which contribute to the sum of the $\Pi_f$ are the forms in $\widetilde{\Ff}_N$. Since $\phi_{\om}$ is nonpositive outside $[-\om,\om]$, removing these nonpositive contributions can only increase the quantity inside an odd power. If we follow the same replacement of the finite-level mean by the limiting mean used in the proof of the first-zero theorem, we can say that
\begin{equation}
\liminf_{\substack{N \to \infty \\ N \text{\rm prime}}} \f{1}{\abs{\Ff_N}}\sum_{f \in \Ff_N}(X_f-\mu)^n \ \geq \ S(n,a;\phi_{\om}).
\end{equation}
Indeed, writing $m_N=\la D(f;\phi_{\om})\ra_{N;+}$, the binomial expansion of
\begin{equation}
D(f;\phi_{\om})-\mu
=
\left(D(f;\phi_{\om})-m_N\right)+(m_N-\mu)
\end{equation}
shows that the terms involving $m_N-\mu$ tend to zero, since $m_N\to\mu$ and the lower centered moments are bounded under the same support condition.

Then let
\begin{equation}
\Delta \ := \ S(n,a;\phi_{\om})-(-\mu)^n.
\end{equation}
Our earlier assumption gives us that $\Delta > 0$. Hence, subtracting $(-\mu)^n$ from the previous mean gives 
\begin{equation}
\liminf_{\substack{N \to \infty \\ N \text{\rm prime}}} \f{1}{\abs{\Ff_N}}\sum_{f \in \Ff_N}\Pi_f \ \geq \ \Delta.
\end{equation}
Therefore, for all sufficiently large prime $N$ we write
\begin{equation}\label{positiveavg}
\f{1}{\abs{\Ff_N}}\sum_{f \in \widetilde{\Ff}_N}\Pi_f \ \geq \ \f{\Delta}{2}.
\end{equation}
Note that the $\Delta/2$ is used only so that the inequality holds for every sufficiently large $N$ rather than only in the limit. Now, we ensure that a very small number of forms cannot make up the entire positive average, so let
\begin{equation}
B_{\om} \ := \ \max_{\abs{x} \ \leq \ \om}\{\phi_{\om}(x)\}.
\end{equation}
Then $0 \leq X_f \leq B_{\om}R_f(\om)$, and expanding $\Pi_f$ yields
\begin{equation}
\Pi_f \ = \ \sum_{i=1}^n\binom{n}{i}X_f^i(-\mu)^{n-i}.
\end{equation}
Further, if $R_f(\om) = 0$ then $\Pi_f = 0$. Otherwise $R_f(\om)\ \geq \ 1$ and each power $X_f^i$ is bounded by some constant times $R_f(\om)^n$. Therefore,
\begin{equation}
\Pi_f^2 \ \ll_{n,\mu,\om} \ R_f(\om)^{2n}.
\end{equation}
By applying Lemma \ref{highmomentbound}, there exists a constant $C > 0$ that is independent of $N$ such that
\begin{equation}\label{secondmomentbound}
\f{1}{\abs{\Ff_N}}\sum_{f\in\Ff_N}\Pi_f^2 \ \leq \ C
\end{equation}
for all sufficiently large prime $N$. Applying Cauchy--Schwarz to equation~\eqref{positiveavg} gives $\abs{\widetilde{\Ff}_N} = a_{\om, N}$, so we obtain

\begin{align}
\f{\Delta}{2} \ &\leq \ \f{1}{\abs{\Ff_N}}\sum_{f\in\widetilde{\Ff}_N}\Pi_f \ \leq \ \left(\f{1}{\abs{\Ff_N}}\sum_{f\in\Ff_N}\mathbbm{1}_{\widetilde{\Ff}_N}(f)\right)^{1/2}\left(\f{1}{\abs{\Ff_N}}\sum_{f\in\Ff_N}\Pi_f^2\right)^{1/2} \ \leq \ \left(\f{a_{\om,N}}{\abs{\Ff_N}}\right)^{1/2}C^{1/2}.
\end{align}
Squaring both sides gives 
\begin{equation}
\f{a_{\om,N}}{\abs{\Ff_N}} \ \geq \ \f{\Delta^2}{4C},
\end{equation}
and since the right-hand side is a positive constant with no dependence on $N$, taking the limit infimum yields our desired result.
\end{proof}

\subsection{A Linear Lower Bound For The Number Of Forms}
Since the family itself has size comparable to $N$, the positive-proportion statement becomes a linear lower bound for the number of forms with a zero in the window.

\linearnumberofforms*

\begin{proof}
From Theorem \ref{positivepercent} we can say that for all sufficiently large prime $N$
\begin{equation}
a_{\om,N} \ \geq \ \f{\Delta^2}{4C}\abs{\Ff_N}.
\end{equation}
Since $\frac{\Delta^2}{4C}$ is a fixed positive constant, we have that $a_{\om,N}\gg_{\om,n}\abs{\Ff_N}$. Moreover, through the standard dimension estimate, we obtain our second claim as $\abs{\Ff_N}\asymp_k N$.
\end{proof}

\subsection{Upper-Tail and Zero-Count Consequences}
The same high-moment estimate also gives information in the opposite direction; forms with a large number of zeros in the window must be rare.

\begin{corollary}\label{zerotail}
Under the same assumptions as Theorem \ref{positivepercent}, there is a constant $C_1>0$ that is independent of $N$ and $r$ such that for every integer $r\geq1$,
\begin{equation}
\limsup_{\substack{N \to \infty \\ N \text{\rm prime}}} \f{\abs{\Set{f \in \Ff_N\mid R_f(\om) \ \geq \  r}}}{\abs{\Ff_N}} \ \leq \ \f{C_1}{r^{2n}}.
\end{equation}
\end{corollary}

\begin{proof}
By Lemma \ref{highmomentbound}, we have that for all sufficiently large prime $N$
\begin{equation}
\f{1}{\abs{\Ff_N}}\sum_{f\in\Ff_N}R_f(\om)^{2n} \ \leq \ C_1.
\end{equation}
With this, we further see that every form with $R_f(\om)\geq r$ contributes at least $r^{2n}$ to this sum and thus
\begin{align}
r^{2n}\abs{\Set{f\in\Ff_N\mid R_f(\om) \ \geq \ r}} \ &\leq \ \sum_{\substack{f\in\Ff_N\\R_f(\om) \ \geq \ r}}R_f(\om)^{2n}
\ \leq \ \sum_{f\in\Ff_N}R_f(\om)^{2n}.
\end{align}
Consequently, if we divide by $r^{2n}\abs{\Ff_N}$ and take the limit supremum then we achieve our desired result.
\end{proof}

We can now combine Theorem \ref{positivepercent} with the upper-tail bound. By choosing the threshold sufficiently large, we obtain a positive proportion of forms whose number of zeros in the window is both nonzero and uniformly bounded.

\begin{corollary}\label{boundedzerocount}
Under the hypotheses of Theorem \ref{positivepercent}, there exists an integer $r_0 \geq 2$ such that
\begin{equation}
\liminf_{\substack{N \to \infty \\ N \text{\rm prime}}}\f{\abs{\Set{f \in \Ff_N\mid 1 \ \leq \ R_f(\om) \ < \ r_0}}}{\abs{\Ff_N}} \ > \ 0.
\end{equation}
\end{corollary}

\begin{proof}
Theorem \ref{positivepercent} gives us the constant $c_{\om} > 0$ such that
\begin{equation}
\liminf_{\substack{N \to \infty \\ N \text{\rm prime}}} \f{\abs{\Set{f \in \Ff_N\mid R_f(\om) \ \geq \ 1}}}{\abs{\Ff_N}} \ \geq \ c_{\om}.
\end{equation}
On the other hand, we have that Corollary~\ref{zerotail} gives us
\begin{equation}
\limsup_{\substack{N \to \infty \\ N \text{\rm prime}}} \f{\abs{\Set{f \in \Ff_N\mid R_f(\om) \ \geq \ r}}}{\abs{\Ff_N}} \ \leq \ \f{C_1}{r^{2n}}.
\end{equation}
Therefore, if we choose a sufficiently large $r_0$ such that $\frac{C_1}{r_0^{2n}} < \frac{c_{\om}}{2}$, we can isolate the forms with at least one (but fewer than $r_0$) zeros by taking all forms with at least one zero and excluding those with $r_0$ or more zeros. Thus,
\begin{align}
\liminf_{\substack{N \to \infty \\ N \text{\rm prime}}} \f{\abs{\Set{f \in \Ff_N\mid 1 \ \leq \ R_f(\om) \ < \ r_0}}}{\abs{\Ff_N}} \ &\geq \ c_{\om}-\f{C_1}{r_0^{2n}} \notag \\
\ &> \ \f{c_{\om}}{2} \ > \ 0.
\end{align}
\end{proof}

\zerocountsubseq*

\begin{proof}
Let $r_0$ be the integer from Corollary~\ref{boundedzerocount}. We have
\begin{equation}
\Set{f \in \Ff_N\mid 1 \ \leq \ R_f(\om) \ < \ r_0} \ = \ \bigcup_{r = 1}^{r_0-1}\Set{f \in \Ff_N\mid R_f(\om) = r},
\end{equation}
where the union is disjoint. The left-hand set has positive lower density while the right-hand union contains only finitely many sets. If every set on the right had upper density zero, their finite union would also have upper density zero, contradicting Corollary~\ref{boundedzerocount}. Thus, at least one integer $r$ has the required positive upper density.
\end{proof}

\section{Acknowledgments}
This research was supported with funding from the National Science Foundation (grant DMS2341670), Missouri State University, University of Rochester, the University of Chicago, the University of Michigan, and Williams College.

\newpage
\section{Notation}\label{notation}
The following table collects the notation used throughout the paper. Unless otherwise stated, $k$ is fixed, $N$ tends to infinity through the primes, and all zeros are counted with multiplicity.
\vspace{0.5cm}

\begingroup
\small
\renewcommand{\arraystretch}{1.18}
\setlength{\tabcolsep}{8pt}
\setlength{\LTpre}{0pt}
\setlength{\LTpost}{0pt}
\begin{longtable}{
@{}
>{\raggedright\arraybackslash}p{0.23\textwidth}
>{\raggedright\arraybackslash}p{0.72\textwidth}
@{}}
\toprule
\textbf{Notation} & \textbf{Meaning} \\
\midrule
\endfirsthead
\toprule
\textbf{Notation} & \textbf{Meaning} \\
\midrule
\endhead
\midrule
\multicolumn{2}{r}{\textit{Continued on the next page}} \\
\endfoot
\bottomrule
\endlastfoot

\multicolumn{2}{@{}l}{\textbf{Modular forms and families}}\\
\addlinespace[2pt]
$k$ & The fixed even weight of the modular forms. \\
$N$ & The level of the family; in the limiting results, $N \to \infty$ through the primes. \\
$H_k^\star(N)$ & The family of normalized holomorphic cuspidal newforms of weight $k$ and level $N$. \\
$H_k^+(N)$, $H_k^-(N)$ & The subfamilies of newforms having root number $+1$ and $-1$, respectively. \\
$\Ff_N$ & The family of even newforms under consideration. \\
$\widetilde{\Ff}_N$ & The subfamily of forms in $\Ff_N$ having at least one normalized zero in $(-\om,\om)$. \\
$\e_f$ & The root number of $f$, defined by $\Lambda(s,f) = \e_f\Lambda(1-s,f)$, with $\e_f \in \{-1,1\}$. \\
$\lm_f(n)$ & The normalized $n$-th Fourier coefficient of $f$. \\
$L(s,f)$ & The $L$-function associated to $f$: $L(s,f)=\sum_{n=1}^{\infty}\lm_f(n)n^{-s}$. \\
$\Lambda(s,f)$ & The completed $L$-function associated to $f$. \\
\addlinespace[6pt]

\multicolumn{2}{@{}l}{\textbf{Zeros and normalized zeros}}\\
\addlinespace[2pt]
$c_f$ & The analytic conductor of $L(s,f)$. \\
$\rho_{f,j}$ & The $j$-th nontrivial zero of $L(s,f)$, written under GRH as $\rho_{f,j}\ = \ \frac12+i\ga_{f,j}$. \\
$\ga_{f,j}$ & The ordinate of the $j$-th nontrivial zero of $L(s,f)$. \\
$\widetilde{\ga}_{f,j}$ & The normalized ordinate of the $j$-th nontrivial zero of $L(s,f)$: $\widetilde{\ga}_{f,j} = \frac{\log c_f}{2\pi}\ga_{f,j}$. \\
$\om$ & The normalized window size. \\
$R_f(\om)$ & The number of normalized zeros of $L(s,f)$ in $[-\om,\om]$, counted with multiplicity $R_f(\om) = \#\Set{j:\abs{\widetilde{\ga}_{f,j}} \leq \om}$. \\

\addlinespace[6pt]
\multicolumn{2}{@{}l}{\textbf{Test functions and Fourier analysis}}\\
\addlinespace[2pt]
$\phi$ & Even Schwartz test function used in the centered-moment argument to weight normalized zeros. \\
$\varphi$ & Even Schwartz test function used to control the number of zeros in a fixed interval. \\
$\widehat{\phi}$ & The Fourier transform of $\phi$. \\
$\supp(\widehat{\phi})$ & The support of the Fourier transform. \\
$C_c^\infty(\R)$ & The space of infinitely differentiable compactly supported functions on $\R$. \\
$\eta$ & A nonzero even nonnegative function in $C_c^\infty(\R)$ supported in $\left(-\frac{1}{4n},\frac{1}{4n}\right)$. \\
$q$ & The inverse Fourier transform of $\eta$. \\
$\delta$ & A positive constant such that $q(x) > 0$ whenever $\abs{x} \leq \delta$. \\
$T$ & A scaling parameter satisfying $T > \max\{\frac{\om}{\delta},1\}$. \\
$q_T$ & The rescaled function $q_T(x) := q(x/T)$. \\
$\widehat{q_T}$ & The Fourier transform of $q_T$, supported in $\left(-\frac{1}{4nT},\frac{1}{4nT}\right)$. \\
$h$ & An even compactly supported auxiliary function on $[-1,1]$, decreasing
on $[0,1]$, from which an optimized test function is constructed. \\
$f(y)$ & The rescaled auxiliary function $f(y) = h(2yn/\si)$. \\
$g = f*f$ & The convolution of $f$ with itself. \\
$\phi_{\om}$ & The sign-changing test function $\phi_{\om}(x) = \ \left(1-\frac{x^2}{\om^2}\right)\widehat{g}(x)$. It is nonnegative for $\abs{x} \leq \om$ and nonpositive for $\abs{x} > \om$. \\
$B_{\om}$ & The maximum value of $\phi_{\om}$ on the zero-counting window: $B_{\om} = \max_{\abs{x} \leq \om}\{\phi_{\om}(x)\}$. \\
$c_{\varphi}$ & The positive lower bound of the auxiliary function $\varphi$ on $[-\om,\om]$: $c_{\varphi} = \min_{\abs{x}\leq\om}\varphi(x)$. \\

\addlinespace[6pt]
\multicolumn{2}{@{}l}{\textbf{Densities, averages, and moments}}\\
\addlinespace[2pt]
$D(f;\phi)$ & The one-level linear statistic $D(f;\phi) = \sum_j\phi(\widetilde{\ga}_{f,j})$. \\
$D_n(f;\Phi)$ & The $n$-level density of $L(s,f)$ with test function $\Phi:\R^n \to \R$. \\
$\la A_f\ra_{N;\pm}$ & The average of a quantity $A_f$ over the family $H_k^\pm(N)$. \\
$\mathcal D_n^\pm(N,\phi)$ & The $n$-th centered moment of the one-level statistic over the even-signed or odd-signed family: $\mathcal D_n^\pm(N,\phi) = \la \left(D(f;\phi)-\la D(f;\phi)\ra_{N;\pm}\right)^n\ra_{N;\pm}$. \\
$\mu(\phi,\Ff)$ & The limiting orthogonal mean: $\mu(\phi,\Ff) = \widehat{\phi}(0) + \frac12\int_{-\infty}^{\infty}\widehat{\phi}(y)\,dy$. When the test function is fixed, this is abbreviated to $\mu$. \\
$V_\phi^2$ & The variance term $V_\phi^2 = 2\int_{-\infty}^{\infty} \abs{y}\widehat{\phi}(y)^2\,dy$. \\
$S(n,a;\phi)$ & The lower-order contribution in the limiting centered-moment formula. \\
$\mathcal R(m,i;\phi)$ & The integral expression appearing in the definition of $S(n,a;\phi)$. \\
$n$ & The order of the centered moment. \\
$a$ & The integer parameter in the centered-moment theorem, satisfying $0\leq a\leq n/2$. \\
$\si$ & The support parameter $\si = \frac{n}{n-a}$ with $1\ \leq \ \si \leq 2$. The support condition is $\supp(\widehat{\phi})\subset\left(-\frac{\si}{n},\frac{\si}{n}\right)$. \\

\addlinespace[6pt]
\multicolumn{2}{@{}l}{\textbf{Quantities in the low-lying-zero bounds}}\\
\addlinespace[2pt]
$a_{\om,N}$ & The number of forms $f\in\Ff_N$ having at least one normalized zero in $(-\om,\om)$: $a_{\om,N} = \abs{\widetilde{\Ff}_N}$. \\
$b_{\om,N}$ & The maximum number of normalized zeros in $(-\om,\om)$ in any one form in $\widetilde{\Ff}_N$. \\
$X_f$ & The contribution from the zeros of $f$ in the window: $X_f = \sum_{\abs{\widetilde{\ga}_{f,j}}\leq\om}\phi_{\om}(\widetilde{\ga}_{f,j})$. \\
$\Pi_f$ & The nonnegative centered contribution $\Pi_f =  (X_f-\mu)^n-(-\mu)^n$ which vanishes whenever $f$ has no normalized zero in $(-\om,\om)$. \\
$\Delta$ & The positive moment gap $\Delta = S(n,a;\phi_{\om})-(-\mu)^n$. \\
$C$ & A constant, independent of $N$, satisfying $\frac{1}{\abs{\Ff_N}}\sum_{f\in\Ff_N}\Pi_f^2 \leq  C$. \\

\addlinespace[6pt]
\multicolumn{2}{@{}l}{\textbf{Additional constants and conventions}}\\
\addlinespace[2pt]
$C_{\varphi,n,\om}$ & A positive constant, independent of $N$, satisfying $\frac{1}{\abs{\Ff_N}}\sum_{f\in\Ff_N}R_f(\om)^{2n} \leq C_{\varphi,n,\om}$ for all sufficiently large prime $N$. \\
$C_1$ & A positive constant, independent of $N$ and $r$, used in the upper-tail estimate $\frac{\abs{\Set{f\in\Ff_N\mid R_f(\om)\geq r}}}{\abs{\Ff_N}} \leq \frac{C_1}{r^{2n}}$. \\
$c_{\om}$ & A fixed positive lower bound for the limiting proportion of forms having at least one normalized zero in $(-\om,\om)$. \\
$r$ & A positive integer used either as a threshold for the number of normalized zeros in the upper-tail estimate or as a fixed exact zero count. \\
$r_0$ & A sufficiently large fixed integer such that a positive proportion of forms satisfy $1\leq R_f(\om)<r_0$. \\
\end{longtable}
\endgroup

\newpage
\nocite{*}
\printbibliography

\end{spacing}
\end{document}